\documentclass[11pt,letterpaper]{article}

\usepackage{amsmath}

\usepackage{amssymb}
\usepackage{amsthm}
\usepackage{mathrsfs}
\usepackage{wasysym}
\usepackage{bbm}		% blackboard 1
\usepackage{mathtools}
\usepackage[shortlabels]{enumitem}
\usepackage{braket}		% defines \middle similar to \left and \right
\usepackage{mdwlist}		% suspend/resume lists
\usepackage{xcolor}
\usepackage{tikz}
\usepackage{graphicx}
\usepackage[a4paper, left=2.7cm, right=2.7cm, top=3.5cm, bottom=2cm]{geometry}
\usepackage[pdftex=true,hyperindex=true,pdfborder={0 0 0}]{hyperref}
\usetikzlibrary{patterns,positioning,arrows,decorations.markings,calc,decorations.pathmorphing,decorations.pathreplacing}
\usepackage{cleveref}

\usepackage[titletoc]{appendix}

\theoremstyle{plain}
\newtheorem{theorem}{Theorem}[section]
\newtheorem{lemma}[theorem]{Lemma}

\newtheorem{proposition}[theorem]{Proposition}

\theoremstyle{definition}
\newtheorem{remark}[theorem]{Remark}
\newtheorem{example}[theorem]{Example}

\newcommand{\infinito}[1]{{}^\infty #1}

\theoremstyle{definition}
\newtheorem{definition}[theorem]{Definition}

\newcommand{\ZZ}{\mathbb{Z}}			% The set of integers
\newcommand{\NN}{\mathbb{N}}			% The set of natural numbers
\newcommand{\diam}{{\rm diam}}

\usepackage[all]{xy}
\usepackage{color}
\usepackage{microtype}

\newcommand{\e}{\varepsilon}
\renewcommand{\d}{\delta}

\newcommand{\vacio}{\tikz[baseline=.1em,scale=.4]{ 
\draw (-.1,0) -- (-.1,.85) -- (.8,.85) -- (.8,0) -- cycle; 
}}

\newcommand{\cero}{\tikz[baseline=.1em,scale=.4]{ 
\draw (-.1,0) -- (-.1,.85) -- (.8,.85) -- (.8,0) -- cycle ; 
\node at (.35,.425) {0};
}}

\newcommand{\uno}{\tikz[baseline=.1em,scale=.4]{ 
\draw (-.1,0) -- (-.1,.85) -- (.8,.85) -- (.8,0) -- cycle; 
\node at (.35,.425) {1};
}}
\newcommand{\dos}{\tikz[baseline=.1em,scale=.4]{ 
\draw (-.1,0) -- (-.1,.85) -- (.8,.85) -- (.8,0) -- cycle; 
\node at (.35,.425) {2};
}}

\newcommand{\cerof}{\tikz[baseline=.0001em,scale=.28]{
  \draw (-.35,-.2) -- (-.35,1.05) -- (1.05,1.05) -- (1.05,-0.2) -- cycle;
  \draw[->] (.95,.85) -- (-.25,.85); 
  \small
  \node at (.35,.325) {0};
}}

\newcommand{\unof}{\tikz[baseline=.0001em,scale=.28]{
  \draw (-.35,-.2) -- (-.35,1.05) -- (1.05,1.05) -- (1.05,-0.2) -- cycle;
  \draw[->] (.95,.85) -- (-.25,.85); 
  \small
  \node at (.35,.325) {1};
}}

\newcommand{\dosf}{\tikz[baseline=.0001em,scale=.28]{
  \draw (-.35,-.2) -- (-.35,1.05) -- (1.05,1.05) -- (1.05,-0.2) -- cycle;
  \draw[->] (.95,.85) -- (-.25,.85); 
  \node at (.35,.325) {2}; 
}}

\newcommand{\flechai}{\tikz[baseline=.1em,scale=.4]{ 
\draw (-.1,0) -- (-.1,.85) -- (.8,.85) -- (.8,0) -- cycle; 
\draw[->] (.7,.425) -- (0,.425);
}}

\title{Local non-periodic order and diam-mean equicontinuity on cellular automata}
\author{Luguis de los Santos Ba\~nos and Felipe Garc\'{\i}a-Ramos}

\date{}	

\begin{document}
\maketitle
\begin{abstract}
 Diam-mean equicontinuity is a dynamical property that has been of use in the study of non-periodic order. Using some type of ``local" skew product between a shift and an odometer looking cellular automaton (CA), we will show there exists an almost diam-mean equicontinuous CA that is not almost equicontinuous, (and hence not almost locally periodic). As an application we show that Kurka's dichotomy does not hold for diam-mean versions of sensitivity and equicontinuity.
 
 Previously we constructed a CA that is almost mean equicontinuous \cite{delossantos} but not almost diam-mean equicontinuous \cite{delossantosautomata}. 
\end{abstract}
\section{Introduction}

The study of non-periodic order, in the sense of finding and classifying order in non-periodic structures, has been a topic of interest among mathematicians, physicists and biologists for many years. For example, in 1940 Hedlund and Morse studied non-periodic sequences with minimal complexity which they called Sturmian sequences \cite{morse1940symbolic}. The role of dynamical systems in this line of research is natural, as it provides several formal ways to define order. 

Cellular automata (CA) were introduced by Ulam and von Neumann to model the evolution of cells. Cellular automata can be studied in the context of \textbf{topological dynamical systems (TDS)}, that is, the study of pairs $(X,T)$, where $X$ is a compact metric space and $T:X\rightarrow X$ a continuous function \cite{kuurka2003topological,ceccherini2010cellular}. CA are sometimes used as discrete models of differential equations \cite{toffoli1984cellular,toffoli1987cellular}. %Beyond applications, cellular automata is of interest in abstract mathematics due to its links with group theory. 

In the study of cellular automata (or symbolic dynamics in general) periodicity is linked to the concept of equicontinuity. A TDS is \textbf{equicontinuous} if the family $\{T^n\}_{n\in\NN}$ is equicontinuous. One may also study this notion locally. A point $x\in X$ is an \textbf{equicontinuity point} if the diameter of the images of a small ball around $x$ will always stay small, that is, if for every $\e>0$ there exists $\d >0$ such that 
$diam(T^{i}B_{\d}(x))<\e$ for every $i\in \NN$. We say a TDS is \textbf{almost equicontinuous} if the equicontinuity points are residual. Given a CA it is not difficult to check that a point is equicontinuous if and only if it is locally eventually periodic; that is, if every column is eventually periodic (Proposition \ref{prop:lep}). For CA, almost equicontinuity is more natural than equicontinuity since it can be used to classify CA using sensitivity to initial conditions (Theorem \ref{thm:dicotomiaCA}).
A weaker notion than equicontinuity, \textbf{diam-mean equicontinuity}, requires the diameter of small balls to stay small on average (see Definition \ref{def:diam-mean}). The notion of diam-mean equicontinuity has been used to characterize regularity properties of the maximal equicontinuous factor \cite{garciajagerye}, which are natural in the context of aperiodic order, and an even weaker property, mean equicontinuity (introduced in \cite{fomin,oxtoby}), has been shown to be connected with the concept of (measure-theoretic) discrete spectrum \cite{weakforms,huang2018bounded,lituye,fuhrmann2018structure} and almost periodic functions \cite{garcia2019mean} (for a survey on mean equicontinuity see \cite{lisurveymean}). Nonetheless, the view point of this paper is not quite the same as in the study of quasicrystals and aperiodic order as in \cite{baake2013aperiodic}, since almost diam-mean equicontinuous systems are only required to be ordered locally almost everywhere (not every point has to be diam-mean equicontinuous), and they may exhibit chaotic properties like positive entropy. 

In \cite{delossantos}, the authors constructed a CA (the Pacman CA) that is almost mean equicontinuous but not almost equicontinuous. It turns out this CA is not almost diam-mean equicontinuous \cite{delossantosautomata}. Hence, the question of whether there exists a CA that is almost diam-mean equicontinuous but not almost equicontinuous remained open. Taking some form of local skew-product between a very regular CA (similar to an odometer) and a very chaotic one (the shift map) we construct a CA that is almost diam-mean equicontinuous but not almost equicontinuous (Theorem \ref{thm:main}). Then, we construct a CA that is neither almost diam-mean equicontinuos nor diam-mean sensitive (Theorem \ref{thm:2}). Note that CA (not necessarily transitive) are almost equicontinuous or sensitive (Kurka's dichotomy \cite{kurka1997languages}).

In summary; almost equicontinuous CA always exhibit eventually periodic behaviour. CA are either almost equicontinuous or sensitive. There exist very chaotic CA, like the shift, which satisfy all the sensitivity-type properties. Nonetheless, among the sensitive CA, there exist almost diam-mean equicontinuous CA, and CA that are neither almost diam-mean equicontinuous nor diam-mean sensitive. Among diam-mean sensitive CA, there exist almost mean equicontinuous, and CA that are neither mean sensitive nor almost mean equicontinuous. 
\section{Preliminaries}
We say $(X,T)$ is a \textbf{topological dynamical system (TDS)} if $X$ is a compact metric space (with metric $d$) and $T:X\rightarrow X$ is a continuous function. Given a metric space $X$, we set $B_{\d}(x)=\{y\in X:d(x,y)<\d\}$, and we denote the diameter of a subset $A$ with $\diam (A)$.
A subset of a topological space is \textbf{residual (or comeagre)} if it includes the intersection of countably many dense open sets.
\begin{definition}

Let $(X,T)$ be a TDS and $x\in X$.  

\begin{enumerate}
\item The point $x$ is an \textbf{equicontinuity point} if for every $\e>0$ there exists $\d >0$ such that $$\diam (T^{i}B_{\d}(x))<\e$$ for every $i\in \NN$.  The set of equicontinuity points of $(X,T)$ is denoted by $EQ$.

\item  $(X,T)$ is \textbf{equicontinuous} if 
$EQ=X$.

\item  $(X,T)$ is \textbf{almost equicontinuous} if $EQ$
is residual.
 
\end{enumerate}
\end{definition}

Diam-mean equicontinuity was introduced in \cite{weakforms} and studied in \cite{garciajagerye}. 

\begin{definition}
\label{def:diam-mean}
	Let $(X,T)$ be a TDS. 
\begin{itemize}

\item We say $x$ is a \textbf{diam-mean equicontinuity point} if for every $\e>0$ there exists $\d >0$ such that 
$$\limsup_{n\to \infty}\frac{\sum_{i=0}^{n} \diam (T^{i}B_{\d}(x))}{n+1}<\e .$$ 
We denote the set of diam-mean equicontinuity points by $EQ$

\item $(X,T)$ is \textbf{diam-mean equicontinuous} if $EQ=X$.

\item $(X,T)$ \textbf{almost diam-mean equicontinuous} if $EQ$ is residual.

\end{itemize}
\end{definition}

It is trivial to see that every equicontinuity point is a diam-mean equicontinuity point.

\begin{proposition}%\cite[Lemma 5]{weakforms}
	\label{pro sigma}

Let $(X,T)$ be a TDS and $\e >0$. We define

$$EQ_{\e}=\{x\in X : \exists \ \d >0, \ \limsup_{n\rightarrow \infty}\frac{\sum_{i=0}^{n}\diam (T^{i}B_{\d}(x))}{n+1}<\e \} .$$
Then, $EQ_{\e}$ is open and $EQ=\bigcap_{m>0} EQ_{\frac{1}{m}}$. Furthermore, $EQ$ is dense if and only if it is a residual set.

\end{proposition}

\begin{proof}
The fact that $EQ_{\e}$ is open follows from \cite[Lemma 46]{weakforms} and \cite[Lemma 4.4]{garciajagerye}. 
%Let $x\in EQ_{\e}$. Therefore, there exists $\d>0$ such that 
%$$ \limsup_{n\rightarrow \infty}\frac{\sum_{i=0}^{n}\diam (T^{i}B_{\d}(x))}{n+1}<\e .$$ 
%Let $y\in B_{\d}(x)$. Let us take $\d'<\min \{ d(x,y),\d -d(x,y) \}$, then $B_{\d'}(y)\subseteq B_{\d}(x)$ and
%$$ \limsup_{n\rightarrow \infty}\frac{\sum_{i=0}^{n}\diam (T^{i}B_{\d}(y))}{n+1}<\e . $$ 
%Hence, $B_{\d}(x)\subseteq EQ_{\e} $. Therefore, $EQ_{\e}$ is open. 

It is easy to check that $EQ=\bigcap_{m>0} EQ_{\frac{1}{m}}$. Thus, by Baire's theorem, $EQ$ is residual if and only if it is dense.
%\begin{itemize}
%\item[$\subseteq$:] Let $x\in EQ$ and $\e>0$. By hypothesis, exists $\d>0$ such that 
%$$\limsup_{n\rightarrow \infty} \frac{\sum_{i=0}^{n}\diam(T^{i}B_{\d}(x))}{n+1}<\e .$$
%Therefore $x\in \bigcap_{m>0} EQ_{\frac{1}{m}}$. 

%\item[$\supseteq$:] Let $x\in \bigcap_{m>0} EQ_{\frac{1}{m}}$. From this we have that for every $\e$ there  exists $\d>0$ such that 
%$$\limsup_{n\rightarrow \infty} \frac{\sum_{i=0}^{n}\diam(T^{i}B_{\d}(x))}{n+1}<\e .$$
%Therefore $x\in EQ$. 
%\end{itemize} 

\end{proof}

Throughout this paper, given $n\in \NN$, we use $[0,n]=\{0,1,...,n\}$.
Now we will give the setup of basic symbolic dynamics. 

%\begin{proof}
%If $x\in X$ is a diam-mean equicontinuous point, then for all $\e>0$ exists $\d>0$ such that 
%$$\limsup_{n\to \infty}\frac{\sum_{i=1}^{n}diam(T^{i}B_{\d}(x))}{n}<\e .$$ 
%Since, all $y\in B_{\d}(x)$ satisfies that $d(T^{i}x,T^{i}y)\leq diam (T^{i}B_{\d}(x))$ for all $i\geq 0$.
%Hence, 
%$$\limsup_{n\to \infty}\frac{\sum_{i=1}^{n}d(T^{i}x,T^{i}y)}{n}<\e ,$$
%for all $i\geq 0$. Therefore, $x$ is a mean equicontinuous point.
%\end{proof}

\begin{enumerate}
\item Given a finite set $A$ (called an alphabet), we define the \textbf{$A$-full shift} as $A^{\ZZ }$. If $X$ is the $A$-full shift for some finite $A$ we say that $X$ is a \textbf{full shift}. 

\item Given $x\in A^{\ZZ}$, we represent the $i$-th coordinate of $x$ as $x_{i}$. Also, given $i,j\in \ZZ$ with $i<j$, we define the finite word $x_{[i,j]}=x_{i}\dots x_{j}$. 

\item We denote the set of all finite words as $A^{+}$.
%Let $i<j\in \ZZ$ and $x\in A^{\ZZ}$. We represent the $ith$ coordinate of $x$ by $x_i$, and we define the word $x_{[i,j]}=x_ix_{i+1}...x_j\in A^{j-i}$.

\item We endow any full shift with the metric 
\begin{displaymath}
\begin{array}{rcl}
d(x,y) & = & \left\{ \begin{array}{ccl} 2^{-i} & \text{if} \ x\neq y & \text{where} \ i=\min \{ |j| : x_{j}\neq y_{j}  \} ; \\  
0 & \text{otherwise.} &
\end{array} \right.
\end{array}
\end{displaymath}

This metric generates the prodiscrete topology $A^{\ZZ}$.
\item For any full shift $A^{\ZZ}$, we define the shift map $\sigma:A^{\ZZ}\rightarrow A^{\ZZ}$ by $\sigma (x)_{i}=x_{i+1}$. The shift map is continuous (with respect to the previously defined metric).

%\item We say $X$ is a \textbf{subshift} (or shift space) if $X \subseteq A^{\ZZ}$ is closed and $\sigma$-invariant.

\end{enumerate}

\begin{definition}\label{defCAShift}
We say that $(X,T)$ is a \textbf{cellular automaton (CA)} if $X$ is a full shift and $T:X\rightarrow X$ is continuous and commutes with $\sigma$, i.e., $\sigma \circ T=T\circ \sigma$. 
\end{definition}

\begin{remark}
	Note that $Tx_i$ represents the $i$th coordinate of the point $Tx$, and $Tx_{[0,n]}$ the word extracted from the $[0,n]$-coordinates of the point $Tx$. 
\end{remark}

The following fact can be extracted from the proof of \cite[Theorem 4]{kurka1997languages}.
\begin{proposition}
	\label{prop:lep}
	Let $(X,T)$ be a CA. If $x\in X$ is an equicontinuity point then $x$ is locally eventually periodic, i.e., for every $i\in \ZZ$ we have that $T^nx_i$ is an eventually periodic sequence (of $n$). 
\end{proposition}

%Clearly, a diam-mean equicontinuity point that is not a equicontinuity point  not necessarily is locally eventually periodic. But, we may find for every $i\in \ZZ$,  non negative intergers $M\geq N$, suficiently large, such that $\{ T^{i}x_{i} : N\leq i \leq M \}$ is an eventually periodic sequence.

%If we have a subset of $\ZZ_{\geq 0}$, the next definition tell us the probability of encountering a element of the desired subset when combing through $\{ 0,\ldots , n\}$ as $n$ grows large.
For CA, there is a simple combinatorial characterization of a diam-mean equicontinuity using upper density. 
\begin{definition}\label{Density}
	
	Let $S\subseteq \NN$. We define the \textbf{upper density} of $S$  by
	$$\overline{D}(S)=\limsup_{n\rightarrow \infty}\frac{\sharp (S\cap [0,n-1] ) }{n}.$$
	
\end{definition}

 Let $n\in \NN$. We will denote the balls of radius $2^{-n}$ with $B_n(x)$. That is, $$B_n(x)=\{y\in A^{\ZZ}:x_i=y_i \forall i\in [-n,n]\}.$$

 Now we will define sensitivity sets on a set of columns. 
\begin{definition}\label{S_J}

Let $J\subset \ZZ$ be finite set and $n\in \NN$. We define 
$$S_{J}(x,n)=\{ i\in \NN : \exists \ y,z\in B_n(x), T^{i}y_{J}\neq T^{i}z_{J} \} .$$

\end{definition}

\begin{proposition}\label{D-M-E-P}
Let $(X,T)$ be a CA and $x\in X$. Then $x$ is a diam-mean equicontinuity point if  and only if for every $m\geq 0$ there exists  $m'\geq 0$ such that  
$$\overline{D}(S_{ \{-j,j \} }(x, m'))\leq \frac{1}{2^{m+2}}$$ for all $0\leq j \leq m+1$.
\end{proposition}

\begin{proof}
\begin{itemize}
\item[$\Rightarrow$:] Suppose there exists $m\geq 0$ such that for all $m'\geq 0$ there exists $l\in[0,m+1]$ such that 
$$
\overline{D}(S_{\{-l,l\} }(x,m'))> \frac{1}{2^{m+2}}.
$$
This implies that
$$
\limsup_{n\rightarrow \infty} \frac{1}{n+1}\sum_{i=0}^{n}\diam (T^{i}B_{m'}(x))
\geq 
\limsup_{n\rightarrow \infty }\frac{1}{n+1}\displaystyle\sum_{i\in S_{\{-l,l\} }(x,m')\cap[0,n]}\diam (T^{i}B_{m'}(x)).
$$

$$
\geq 
\frac{1}{2^l}\limsup_{n\rightarrow \infty }\frac{1}{n+1}\sharp (S_{\{-l,l\} }(x,m')\cap [0,n] )
$$
$$
\geq 
\frac{1}{2^{m+1}}
\overline{D}(S_{\{-l,l\} }(x,m'))> \frac{1}{2^{2m+3}}.
$$

Therefore, $x$ is not a diam-mean equicontinuity point.

\item[$\Leftarrow$:] 
For every pair of integers $n,k\in \NN$ and every $x\in X$ we define the set
$$
S_{k}^{n}(x,m)=S_{[-k,k]}(x,m)\cap [0,n].
$$
Note that for every $k$ we have that
$$
S_{k}^{n}(x,m)\subseteq S_{k+1}^{n}(x,m)\text{, }
$$
and
$$
S_{k+1}^{n}(x,m)\setminus S_{k}^{n}(x,m)=S_{k}^{n}(x,m)\cap [0,n].
$$

Now, let us assume that for every $m\geq 0$ there exists $m'\geq 0$ such that  $$\overline{D}(S_{ \{-j,j \} }(x,m'))\leq \frac{1}{2^{m}}$$ for every $0\leq j\leq m+1$. 
For sufficiently large $m$ we conclude that
\small
\begin{displaymath}
\begin{array}{rl}
 & \displaystyle\limsup_{n\rightarrow \infty} \frac{1}{n+1}\displaystyle\sum_{i=0}^{n}\diam (T^{i}B_{m'}(x))
 \\
 \leq
 &\displaystyle\limsup_{n\rightarrow \infty }\frac{1}{n+1}\left[ \sharp (S^{n}_0(x,m'))+\displaystyle\sum_{i=1}^{\infty}\frac{1}{2^{i}}\sharp (S^{n}_i(x,m')\setminus S^{n}_{i-1}(x,m')) \right] 
\\
=
&\overline{D}(S_{ \{0\} }(x,m'))+\displaystyle\sum_{i=1}^{\infty}\frac{1}{2^{i}}\overline{D}(S_{ \{-j,j \} }(x,m'))
\\
\leq & \displaystyle\frac{1}{2^m}+
\sum_{i=1}^{m+1}\frac{1}{2^{i}}\cdot\frac{1}{2^m}+
\sum_{i=m+2}^{\infty}\frac{1}{2^{i}}
\\
\leq & \displaystyle\frac{1}{2^{m-3}}.
\end{array}
\end{displaymath}

\end{itemize}
This implies $x$ is a diam-mean equicontinuity point.
\end{proof}

\section{Almost diam-mean equicontinuity}

As we mentioned in the previous section, every almost equicontinuous CA is almost diam-mean equicontinuous. In this section we will construct an almost diam-mean equicontinuous CA that is not almost equicontinuous. For other examples of CA related to these concepts see \cite{delossantos,torma2015uniquely}.

%\begin{proof}
%\begin{itemize}

%\item[$\Rightarrow$:]

%Let $x\in A^{\ZZ}$ be a diam-mean equicontinuous point. By Proposition \ref{D-M-E-P->M-E-P} $x$ is a mean equicontinuous point. By Proposition \ref{M-E-P} we ahve that for all $m\geq 0$ exists $m_{0}\geq 0$ such that $y \in B_{2^{-m_{0}}}(x)$ satisfies that
%$$\overline{D}(S_{y,j}^{M})\leq \frac{1}{2^{m+2}},$$ 
%for all $0\leq j \leq m+1$.

%Now, since $x$ is a diam-mean equicontinuous point, then exists a $m_{1}\geq 0$ such that
%$$\limsup_{n\to \infty}\frac{\sum_{i=1}^{n}diam(T^{i}B_{2^{-m_{1}}}(x))}{n}<2^{-m} .$$
%Hence, for all $\e$ exists a finite set $N\subset \ZZ_{0\geq}$ such that for all $l\in N$ 
%$$\frac{\sum_{i=1}^{l}diam(T^{i}B_{2^{-m_{1}}}(x))}{l}\geq 2^{-m}+\e$$

%Observe that 
%$$S_{j}=\bigcup_{y\in B_{2^{-m_{1}}}} S_{y,j}^{M}.$$
%We have that
%$$\limsup_{n\to \inf}\frac{\sharp(S_{j}\cup [0,n])}{n+1}= \limsup_{n\to \infty}
%\frac{\sharp(\bigcup_{y\in B_{2^{-m_{1}}}} (S_{y,j}^{M}\cap [0,n]))}{n+1}$$
%$$\leq\limsup_{n\to \infty}\frac{\sharp (S_{y,j}^{M}\cap [0,n])}{n+1}\leq 2^{-(m+2)}$$

%\item[$\Leftarrow$:] Let $x\in A^{\ZZ}$ such that for every $m\geq 0$ there exists  $m'\geq 0$ that 
%satisfies 
%$$\overline{D}(S_{j})\leq \frac{1}{2^{m+2}},$$ for all $0\leq j \leq m+1$.
%\end{itemize}
%\end{proof}

First, we define a CA that resembles an odometer. 
Let $A_{1}=\{  \vacio \}\cup Z_3=\{  \vacio , \cero, \uno, \dos \}$. We define the CA $T_{1}:A_{1}^{\ZZ}\rightarrow A_{1}^{\ZZ}$ locally as follows: $T_{1}x_i=\vacio$ if and only if $x_i=\vacio$, otherwise $T_{1}x_i\in \{x_i,(x_i+1)$mod$3\}$, with $T_{1}x_i=(x_i+1)$mod$3$ if and only if $x_{i+1}\in \{\vacio,\dos\}$. In other words

\begin{displaymath}
\begin{array}{ccc}
T_{1}x_{i} & = & \displaystyle\left\{ \begin{array}{ccl}
\vacio & if & x_{i} = \vacio \\

%\flechai & if & x_{i}\in \{ \vacio , \flechai  \} \wedge x_{i+1}\in \{ \flechai, \dosf \}  ;\\

\cero & if & (x_{i} = \dos  \wedge x_{i+1}\in \{ \vacio , \dos \} ) \vee \\
 & & (x_{i}=\cero \wedge x_{i+1}\in A_{1}\setminus \{ \vacio , \dos  \} ) ;\\

\uno & if & (x_{i}= \cero \wedge x_{i+1}\in \{ \vacio , \dos \} ) \vee \\
 & &(x_{i}=\uno \wedge x_{i+1}\in A_{1}\setminus \{ \vacio  , \dos \} ) ;\\

\dos & if & x_{i}=\uno \wedge x_{i+1}\in \{ \vacio , \dos \} ) \vee \\
 & & (x_{i}=\dos \wedge x_{i+1}\in A_{1}\setminus \{ \vacio ,\dos  \} ) .\\

%\cerof & if & (x_{i}\in \{ \dos ,\dosf\} \wedge x_{i+1}\in \{ \flechai,\dosf \} \\
% & & \vee (x_{i}=\cerof\wedge x_{i+1}\in A\setminus \{ \vacio ,\flechai,\dos ,\dosf\} );\\

%\unof & if & (x_{i}= \cero \wedge x_{i+1}\in  \{ \flechai, \dosf \} ) \vee  \\
% & & (x_{i}=\cerof \wedge x_{i+1}\in \{ \vacio ,\flechai,\dos ,\dosf \} ) \vee \\ 
% & & (x_{i}=\unof \wedge x_{i+1}\in A\setminus \{ \vacio  , \flechai,\dos , \dosf \});\\

%\dosf & if & (x_{i}= \uno \wedge x_{i+1}\in  \{ \flechai, \dosf \} ) \vee  \\ 
% & & (x_{i}=\unof \wedge x_{i+1} \in \{ \vacio  , \flechai,\dos , \dosf \} ). 
%\vee (x_{i}=\f{blue} \wedge x_{i+1}\in A\setminus \{ \vacio  , \flechai,\dos , \dosf \});\\ .\\
\end{array}\right.

\end{array}
\end{displaymath}

\begin{example}\label{4puertas}
Let $x\in A_{1}^{\ZZ}$ such that $x_{[0,5]}=\vacio \cero \cero \cero \cero \vacio $. We have that

	\begin{displaymath}
	\begin{array}{cccccccccccccccccc}
T_{1}^{0}x_{[0,5]} & = &  \vacio & \cero   &\cero   &\cero    &\cero        &\vacio  & & & T_{1}^{14}x_{[0,5]} & = &\vacio &	\uno   &\cero   &\uno    &\dos        &\vacio\\
T_{1}^{1}x_{[0,5]} & = &\vacio &	\cero   &\cero   &\cero    &\uno         &\vacio  & & & T_{1}^{15}x_{[0,5]} & = &\vacio &	\uno   &\cero   &\dos   &\cero         &\vacio\\
T_{1}^{2}x_{[0,5]} & = &\vacio &	\cero   &\cero   &\cero    &\dos        &\vacio  & & & T_{1}^{16}x_{[0,5]} & = &\vacio &	\uno   &\uno   &\dos   &\uno         &\vacio\\
T_{1}^{3}x_{[0,5]} & = &\vacio &	\cero   &\cero   &\uno    &\cero         &\vacio  & & & T_{1}^{17}x_{[0,5]} & = &\vacio &	\uno   &\dos  &\dos   &\dos        &\vacio \\
T_{1}^{4}x_{[0,5]} & = &\vacio &	\cero   &\cero   &\uno    &\uno         &\vacio  & & & T_{1}^{18}x_{[0,5]} & = &\vacio & \dos  &\cero   &\cero    &\cero         &\vacio\\
T_{1}^{5}x_{[0,5]} & = &\vacio &	\cero   &\cero   &\uno    &\dos        &\vacio  & & & T_{1}^{19}x_{[0,5]} & = &\vacio &	\dos  &\cero   &\cero    &\uno         &\vacio\\
T_{1}^{6}x_{[0,5]} & = &\vacio &	\cero   &\cero   &\dos   &\cero         &\vacio  & & & T_{1}^{20}x_{[0,5]} & = &\vacio &	\dos  &\cero   &\cero    &\dos        &\vacio\\
T_{1}^{7}x_{[0,5]} & = &\vacio &	\cero   &\uno   &\dos   &\uno         &\vacio  & & & T_{1}^{21}x_{[0,5]} & = &\vacio &	\dos  &\cero   &\uno    &\cero         &\vacio\\
T_{1}^{8}x_{[0,5]} & = &\vacio &	\cero   &\dos  &\dos   &\dos        &\vacio  & & & T_{1}^{22}x_{[0,5]} & = &\vacio &	\dos  &\cero   &\uno    &\uno         &\vacio \\
T_{1}^{9}x_{[0,5]} & = &\vacio & \uno   &\cero   &\cero    &\cero         &\vacio  & & & T_{1}^{23}x_{[0,5]} & = &\vacio &	\dos  &\cero   &\uno    &\dos        &\vacio \\
T_{1}^{10}x_{[0,5]} & = &\vacio &	\uno   &\cero   &\cero    &\uno     &\vacio  & & & T_{1}^{24}x_{[0,5]} & = &\vacio &	\dos  &\cero   &\dos   &\cero         &\vacio \\
T_{1}^{11}x_{[0,5]} & = &\vacio &	\uno   &\cero   &\cero    &\dos    &\vacio  & & & T_{1}^{25}x_{[0,5]} & = &\vacio &	\dos  &\uno   &\dos   &\uno         &\vacio\\
T_{1}^{12}x_{[0,5]} & = &\vacio &	\uno   &\cero   &\uno    &\cero     &\vacio  & & & T_{1}^{26}x_{[0,5]} & = &\vacio &	\dos  &\dos  &\dos   &\dos        &\vacio \\
T_{1}^{13}x_{[0,5]} & = &\vacio &	\uno   &\cero   &\uno    &\uno     &\vacio  & & & T_{1}^{27}x_{[0,5]} & = &  \vacio & \cero   &\cero   &\cero    &\cero & \vacio. \\
 
	\end{array}
	\end{displaymath}

\end{example}

\begin{remark}  \label{rem0}
Let $x,y\in A_{1}^{\ZZ}$. Note that $T_{1}x_i$ only depends on $x_i$ and $x_{i+1}$. Hence,  if there exist $M,N\in \ZZ$ such that $x_{M+i}=y_{N+i}$ for every $i\geq 0$, then $T_{1}x_{M+i}=T_{1}y_{N+i}$ for every $i\geq 0$.
\end{remark}

\begin{remark}
	\label{rem}
It is not difficult to see that $T_1$ is almost equicontinuous. In fact, $\vacio$ is a \textbf{blocking word}; that is: if $x,y\in A_1^{\ZZ}$ with $x_{[m',m]}=y_{[m',m]}$ and $x_m=\vacio$, then 
$T_{1}^nx_{[m',m]}=T_{1}^ny_{[m',m]}$ for every $n>0$. 	Moreover, in this situation, one can check that there exist $M>0$ and $p>0$ such that $T_{1}^{M+ip}x_{[m',m]}=T_{1}^{M}x_{[m',m]}$ for all $i\geq 0$.
\end{remark}

Given a CA $(X,T)$, we say a point $x\in X$ is \textbf{periodic} with period $p$ if $T^px=x$. This should not be confused with the statement. Let $J\subset \ZZ$. We say $x_J$ is \textbf{periodic}, if the sequence $\{T^n_J\}_{n\in\NN}$ is periodic.  If $x$ is periodic then $x_J$ is periodic for every $J\subset \ZZ$. The converse may not hold (the periodic can increase).  

We will now present some statements that will help us to understand the behavior of $T_1$. 

\begin{lemma}\label{uno l=0}

Let $x\in A_{1}^{\ZZ}$ such that $x_{[0,1]}= \textnormal{\cero \vacio}$. Then $x_{[0,1]}$ is periodic with period $3$.

\end{lemma}

\begin{proof}
Observe that $T_{1}x_{i}=\vacio$ if and only if $x_{i}=\vacio$. Hence, 

\begin{itemize}
\item $T_{1}x_{0}=\uno$ and $T_{1}x_{1}=\vacio$;
\item $T_{1}^{2}x_{0}=\dos$ and $T_{1}^{2}x_{1}=\vacio$; and
\item $T_{1}^{3}x_{0}=\cero$ and $T_{1}^{3}x_{1}=\vacio$.
\end{itemize}
Therefore, $x_0$ has period $3$.
\end{proof}
\begin{remark} 
Using that $T_1$ commutes with the shift we obtain that for every $z\in\ZZ$ if $x_{[z,z+1]}= \textnormal{\cero \vacio}$, then $x_{[z,z+1]}$ is periodic with period $3$.
\end{remark}
\begin{lemma}\label{uno l=1}

Let $x\in A_{1}^{\ZZ}$ such that $x_{[0,2]}= \textnormal{\cero \cero \vacio}$. Then $x_{[0,2]}$ has period $9$.

\end{lemma}

\begin{proof}
 Remark \ref{rem} and Lemma \ref{uno l=0} implies that $T_{1}^{3k}x_{1}=1$ and $T_{1}^{k}x_{2}=\vacio$ for all $k\in \ZZ_{\geq 0}$. Thus, we have that
\begin{itemize}
\item $T_{1}^{3}x_{[0,1]}=\uno \ \cero$ and $T_{1}^{3}x_{2}=\vacio$;
\item $T_{1}^{6}x_{[0,1]}=\dos \ \cero$ and $T_{1}^{6}x_{2}=\vacio$;
\item $T_{1}^{9}x_{[0,1]}=\cero \ \cero$ and $T_{1}^{9}x_{2}=\vacio$.
\end{itemize}
Therefore, $x_{[0,2]}$ have period $9$.
\end{proof}

From the proof of Lemma \ref{uno l=1}, we can conclude the next result.

\begin{lemma}\label{uno,dos,tres l=1}

If $x\in A_{1}^{\ZZ}$ such that $x_{[0,2]}= \textnormal{\cero \cero \vacio}$, then
\begin{enumerate}
\item $T_{1}^{i}x_{0}=\textnormal{\cero}$ for all $0\leq i\leq 2$;
\item $T_{1}^{i}x_{0}=\textnormal{\uno}$ for all $3\leq i\leq 5$;
\item $T_{1}^{i}x_{0}=\textnormal{\dos}$ for all $6\leq i\leq 8$.
\end{enumerate}
\end{lemma}

Something to be careful about is that the period does not necessarily increase if the amount of $\cero$s increases (one of the differences with an odometer). The following lemma is an evidence of this comment.

\begin{lemma}\label{3 unos}
Let $x,y\in A_{1}^{\ZZ}$ such that $x_{[0,2]}= \textnormal{\cero \cero \vacio}$ and $y_{[0,3]}= \textnormal{\cero \cero \cero \vacio}$. Then $x_{[0,2]}$ and $y_{[0,3]}$ have period $9$.
\end{lemma}

\begin{proof}
By Lemmas \ref{uno l=1} and \ref{uno,dos,tres l=1} we have that 
$$T_{1}^{i}y_{0}=\cero,$$
for all $0\leq i\leq 6$. Observe that
$$T_{1}^{6+i}y_{0}=T_{1}^{i}y_{2},$$
for all $0\leq i\leq 3$. Hence, $T_{1}^{9}y_{[0,3]}=\cero \ \cero \ \cero \ \vacio$.
\end{proof}

To generalize Lemma \ref{3 unos} first we need the following statement.

%\begin{lemma}
%	\label{lem:blockingword}
%If $w\in A^{+}$ is such that $w_{|w|-1}=\vacio$, then $w$ is a $1$-blocking word such that for all $y\in [w]_{0}$ we have that 
%$$T_{1}^{i}x_{[0,|w|)}=T_{1}^{i}y_{[0,|w|)},$$
%for all $i\in \ZZ_{0\geq}$.	
%\end{lemma}

%\begin{proof}
%Let us assume that there exists $y\in [w]_{0}$ such that 
%$$T_{1}^{i}x_{l}\neq T_{1}^{i}y_{l},$$
%for some $i>0$ and $0\leq l <|w|$.
%Clearly $x_{l} \neq \vacio \neq y_{l}$.
%\end{proof}

\begin{lemma}
	\label{lem:basico}
	Let $m,k>0$, $x,y\in A_1^{\ZZ}$ such that $x_{[0,k]}=y_{[0,k]}$, and $\{T_1^ix_{k+1},T_{1}^iy_{k+1}\}\subset \{\vacio,\textnormal{\dos}\}$ for every $i\in[0,m]$. Then $T_{1}^ix_{[0,k]}=T_{1}^iy_{[0,k]}$ for every $i\in[0,m]$.  
\end{lemma}

\begin{proof}
%Since $\{T_1^ix_{k+1},T^iy_{k+1}\}\subset \{\vacio,\dos\}$ for every $i\in[0,m]$, then we have the following cases:
%\begin{enumerate}
%\item $T_1^ix_{k+1}=\vacio =T^iy_{k+1}$, for every $i\in[0,m]$;
%\item $T_1^ix_{k+1}=\dos =T^iy_{k+1}$, for every $i\in[0,m]$;
%\item $T_1^ix_{k+1}=\vacio$ and $T^iy_{k+1}=\dos$, for every $i\in[0,m]$, and;
%\item $T_1^ix_{k+1}=\dos$ and $T^iy_{k+1}=\vacio$, for every $i\in[0,m]$.
%\end{enumerate}
%For case 1 we have that, by Remark \ref{rem}, $T^ix_{[0,k]}=T^iy_{[0,k]}$ for every $i\in[0,m]$. 

%Now, since the radius is $1$ and is to the right, if the case 2 happens, then  $T^ix_{[0,k]}=T^iy_{[0,k]}$ for every $i\in[0,m]$.

%In case 3, we have that $T^ix_{k}=T^iy_{k}$ for every $i\in[0,m]$.

%Since the radius is $1$ and is to the right, $x_{k}=y_{k}$ and $\{ T_1^ix_{k+1},T^iy_{k+1}\}\subset \{\vacio,\dos\}$ for every $i\in[0,m]$, we have that $T^ix_{k}=T^iy_{k}$ for every $i\in[0,m]$. By this, since the radius is $1$ and is to the right, $x_{[0,k]}=y_{[0,k]}$, then $T^ix_{[0,k]}=T^iy_{[0,k]}$ for every $i\in[0,m]$.

The proof can be obtained using Remark \ref{rem0}, the shift commuting property of a CA, and the fact that if $x',y'$ satisfies that $x'_0=y'_0$ and $\{x'_{1},y'_{1}\}\subset \{\vacio,\dos\}$, then $T_1x'_0=T_1y'_0$.

\end{proof}

Now we will prove more important properties of $T_1$.
\begin{proposition}\label{lemaNpuertas}

Let $l\in \NN$, $j\in [0,2^{l}]$, and $x,y\in A_{1}^{\ZZ}$ with $x_{[0,2^{l}]}=\textnormal{\cero}^{2^{l}} \vacio$ and $y_{[0,2^{l}+j]}= \textnormal{\cero}^{2^{l}+j} \vacio$. We have that $x_{[0,2^{l}]}$ and $y_{[0,2^{l}+j]}$ have period $3^{l+1}$. Furthermore, 
      \begin{itemize}
      \item $T_{1}^{i}x_{0}=\textnormal{\cero}$ for all $0\leq i <3^{l}$;
      \item $T_{1}^{i}x_{0}=\textnormal{\uno}$ for all $3^{l}\leq i <2(3^{l})$;
      \item $T_{1}^{i}x_{0}=\textnormal{\dos}$ for all $2(3^{l})\leq i <3^{l+1}$.
      \end{itemize}

\end{proposition}

\begin{proof} 
We will prove this result using induction on $l$. From Lemma \ref{uno l=1}, we have the result for $l=0$. Let us assume that the results hold for $l=k$; that is:
\begin{equation}\label{Proof:l=k}
\begin{array}{rcl}
  T_{1}^{i}x_{0}&=&\textnormal{\cero \ for all } 0\leq i <3^{k};\\
  T_{1}^{i}x_{0}&=&\textnormal{\uno \ for all } 3^{k}\leq i <2(3^{k});\\
  T_{1}^{i}x_{0}&=&\textnormal{\dos \ for all } 2(3^{k})\leq i <3^{k+1};\\
  x_{[0,2^{k}]} & \text{and} & y_{[0,2^{k}+j]} \text{ have period } 3^{k+1}.
\end{array}
\end{equation}
Now, let $l=k+1$. By \eqref{Proof:l=k} and Remark \ref{rem} we have that
      \begin{itemize}
      \item $T_{1}^{i}x_{2^{k}}=\cero$ for all $0\leq i <3^{k}$;
      \item $T_{1}^{i}x_{2^{k}}=\uno$ for all $3^{k}\leq i <2(3^{k})$;
      \item $T_{1}^{i}x_{2^{k}}=\dos$ for all $2(3^{k})\leq i <3^{k+1}$.
      \end{itemize}
Hence, we have that $T_{1}^{i}x_{[0,2^{k})}=\cero^{2^{k}}$ for all $0\leq i \leq 2(3^{k})$. Observe  that 
$$T_{1}^{2(3^{k})}x_{[0,2^{k}]}= \cero^{2^{k}} \dos \text{ and }  x_{[2^{k},2^{k+1}]}=\cero^{2^{k}} \vacio.$$
Using Lemma \ref{lem:basico} and the fact that $T_1$ commutes with the shift, we obtain that
$$T_{1}^{2(3^{k})+i}x_{[0,2^{k})}=T_{1}^{i}x_{[2^{k},2^{k+1})} ,$$
for all $0\leq i < 3^{k}$. This implies that  $T_{1}^{i}x_{0}=\cero$ for all $0\leq i <3^{k+1}$. 

By \eqref{Proof:l=k} we have that $y_{[0,2^{k+1}-1]}=\cero^{2^{k+1}-1} \vacio$ has period $3^{k+1}$. Since $x_{(0,2^{k+1}]}=y_{[0,2^{k+1}-1]}$, Remark \ref{rem} gives us that 
$$T_{1}^{3^{k+1}}x_{(0,2^{k+1}]}=x_{(0,2^{k+1}]}=\cero^{2^{k+1}-1}\vacio.$$
 %Then, by the Remark \ref{rem}, we have that $T$ 
Thus,
 $$T_{1}^{3^{k+1}-1}x_{[0,2^{k+1}]}=\cero \ \dos^{2^{k+1}-1} \vacio \text{, and}$$  
$$T_{1}^{3^{k+1}}x_{[0,2^{k+1}]}=\uno \ \cero^{2^{k+1}-1} \vacio.$$ 
By this, and a similar use of \eqref{Proof:l=k}, we obtain that \begin{itemize}
      \item $T_{1}^{3^{k+1}+i}x_{2^{k}}=\cero$ for all $0\leq i <3^{k}$;
      \item $T_{1}^{3^{k+1}+i}x_{2^{k}}=\uno$ for all $3^{k}\leq i <2(3^{k})$;
      \item $T_{1}^{3^{k+1}+i}x_{2^{k}}=\dos$ for all $2(3^{k})\leq i <3^{k+1}$.
\end{itemize}
Then, $T_{1}^{3^{k+1}+i}x_{[0,2^{k})}=\uno \ \cero^{2^{k}-1}$ for all $0\leq i \leq 2(3^{k})$. So, using Lemma \ref{lem:basico} and the fact that $T_1$ commutes with the shift, we have that
$$T_{1}^{3^{k+1}+i}x_{[0,2^{k})}=T_{1}^{i}x_{[2^{k},2^{k+1})}$$
for all $2(3^{k})\leq i < 3^{k+1}$. Therefore, $T_{1}^{i}x_{0}=\uno$ for all $3^{k+1}\leq i <2(3^{k+1})$.
In a similar way we have that $T_{1}^{i}x_{0}=\dos$ for all $2(3^{k+1})\leq i <3^{k+2}$.
Hence,
\begin{itemize}
      \item $T_{1}^{i}x_{0}=\cero$ for all $0\leq i <3^{k+1}$;
      \item $T_{1}^{i}x_{0}=\uno$ for all $3^{k+1}\leq i <2(3^{k+1})$; and
      \item $T_{1}^{i}x_{0}=\dos$ for all $2(3^{k+1})\leq i <3^{k+2}$.
\end{itemize}
Therefore, $x_{[0,2^{k+1}]}$ has period $p=3^{k+2}$. Then, for all $l\in \NN$ we have that $x_{[0,2^{l}]}=\cero^{2^{l}}\vacio$ satisfies:
\begin{itemize}
      \item $T_{1}^{i}x_{0}=\textnormal{\cero}$ for all $0\leq i <3^{l}$;
      \item $T_{1}^{i}x_{0}=\textnormal{\uno}$ for all $3^{l}\leq i <2(3^{l})$;
      \item $T_{1}^{i}x_{0}=\textnormal{\dos}$ for all $2(3^{l})\leq i <3^{l+1}$.
      \end{itemize} 

Now, let $0\leq j < 2^{k+1}$. Using the induction hypothesis we have that $T_{1}^{i}y_{[0,j)}=\cero^{j}$ for all $0\leq i \leq 2(3^{k+1})$. Using Lemma \ref{lem:basico} and the fact that $T_1$ commutes with the shift, we obtain that 
$$T_{1}^{2(3^{k+1})+i}y_{[0,j)}= T_{1}^{i}x_{[2^{k+1}-j,2^{k+1})}$$
for all $0\leq i\leq 3^{k+1}$. Since $x$ has period $3^{k+1}$ we have that
$$T_{1}^{3^{k+1}}x_{[2^{k+1}-j,2^{k+1}]}=\cero^{j}\vacio.$$
Therefore, $y_{[0,2^{k+1}+j]}$ has period $2(3^{k+1})+3^{k+1}=3^{k+2}$.
\end{proof}

% \begin{lemma} \label{lemaNpuertas}

%Let $x\in A_{1}^{\ZZ}$ and $l\geq 0$. If $x_{[0,2^{l}]}=\cero^{2^{l}} \vacio$, then $ T_{1}^{3^{l+1}j-k}x_0 = $\dos \ for all $1\leq j$ and $1\leq k\leq 3^{l}$.

%\end{lemma}

%\begin{proof}
%We will prove this by induction. 	
%The case $l=0$ can be checked directly because $T^nx_0$ has period 3. Now, let us assume that for $l=p$ the hypothesis of the lemma is true. Let $l=p+1$. Since it is true for $k=p$ we have that:
%\begin{itemize}
%\item $T_{1}^{3^{p+1}j-k}x_{2^{p}} = \cero$ for all $2(3^{p})+1 \leq k \leq 3(3^{p})$ and all $j\geq 1$,
%\item $T_{1}^{3^{p+1}j-k}x_{2^{p}} = \uno$ for all $3^{p}+1 \leq k \leq 2(3^{p})$ and all $j\geq 1$, 
%\item and $ T_{1}^{3^{p+1}j-k}x_{2^{p}} = \dos$ for all $1\leq k\leq 3^{p}$ and all $1\leq j$ .
%\end{itemize} 
 %Hence we have that $ T_{1}^{3^{l+1}j}x_{[1,2^{p+1})} = \cero$ for all $1\leq j$. Therefore, $T_{1}^{3^{l+1}j-k}x_0 = \dos$ for all $1\leq j$ and all $1\leq k\leq 3^{l}$. 
     
%\end{proof}

%Let $A_{2}=\{ \vacio , \flechai \}$, and $\Phi:A^3_{2}\rightarrow A_2$ the local rule of the left traffic rule (the reflection of rule 184). 

\begin{proposition} \label{Siempre hay una salida}
	If $x\in A_1^{\ZZ}$ is such that $x_{j}=\vacio$ for some $j\in \ZZ$, then for all $i\in \NN$ we have $T_{1}^n x_{j-i}\in \{\vacio,\textnormal{\dos}\}$ for infinitely many $n>0$.
\end{proposition}
\begin{proof}
We will prove this result by induction on $i$. The result follows for $i=0$ using straightforward applications of the rules of the automaton (as in Lemma \ref{uno l=0}). 
%From Lemma \ref{uno l=0}, we have that $T_{1}^{n}x_{j-1}=\dos$, where $0\leq n\leq 2$. Let us assume that there exists $n> 0$ such that $T_{1}^{n} x_{j-l}=\dos$, where $l>0$. 
%If $T_{1}^{n} x_{j-l}= \vacio$, then $x_{j-l}=\vacio$.Hence, $T_{1}^{n}x_{j-(l+1)}=\dos$, where $0\leq n\leq 2$. So, let us assume that $T_{1}^{n} x_{j-l}= \dos$. 
Assume the result holds for $i>0$. We may assume that $x_{j-i}\neq \vacio$ and $x_{j-i-1}\neq \vacio$; otherwise, the result is straightforward. Hence, we have that $\{ n>0 : T_{1}^{n} x_{j-i}= \dos \}$ is infinite. Furthermore, by the rules of the automaton for every $n' \in \{ n>0 : T_{1}^{n} x_{j-i}= \dos \}$ we have that $T^{n'+1}x_{j-i-1}=(T^{n'}x_{j-i-1}+1)$mod$3$. With this we can conclude the result. 

\end{proof}
 
Now we will combine $T_1$ with the shift map.  
Let $A_{2}=\{ \vacio , \flechai \}$, and $\sigma:A_{2}^{\ZZ}\rightarrow A_{2}^{\ZZ}$ the shift map.

Let $A=A_1\times A_2$. At times we will identify $A$ with the following set $$A=\{ \vacio ,\cero ,\uno ,\dos,\flechai ,\cerof ,\unof,\dosf \}.$$
Note that with this notation, we identify the point $\vacio\times\vacio\in A$ with $\vacio$. In general it will be clear which $\vacio$ we are referring to. 
 Let $\gamma_{1}: A\rightarrow A_{1}$ and $\gamma_{2}: A\rightarrow A_{2}$ be the projection functions. We also extend such functions to $A^{\ZZ}\rightarrow A^{\ZZ}_{1}$ and $A^{\ZZ}\rightarrow A^{\ZZ}_{2}$, respectively.

%\begin{itemize}
%\item $\gamma_{1}: A\rightarrow A_{1}$ as $\gamma_{1}(x)_{i}=\xi$ if and only if $x_{i}=\shf$, where $\xi \in A_{1}$;

%\item $\gamma_{2}: A\rightarrow A_{2}$ as $\gamma_{2}(x)_{i}=\flechai$ if and only if $x_{i}\in \{ \flechai , \cerof , \unof , \dosf  \}$;
%\end{itemize}

We define the CA $T:A^{\ZZ}\rightarrow A^{\ZZ}$ locally (and coordinate-wise) as the only CA that satisfying:
$$
(\gamma_1 (Tx))_i=(T_1\gamma_1 (x))_i \text{, and}
$$ 
$$
(\gamma _{2}(Tx))_{i}=\left\{ 
\begin{array}{cc}
	(\sigma \gamma _{2}(x))_{i} & \text{if }\left\{ \flechai,
	\dosf \right\} \cap \left\{ x_{i},x_{i+1}\right\} \neq \emptyset 
	\text{,} \\ 
	(\gamma _{2}(x))_{i} & \text{otherwise.}
\end{array}
\right. 
$$

%\begin{displaymath}
%\begin{array}{ccc}
%Tx_{i}=\gdos & \Leftrightarrow & x_{[i-1,i]}\in \{ \gmenosuno \ \flechai , \gmenosuno \dosf \} \vee x_{[i,i+1]} \in \{ \guno \ \flechai , \guno \ \dosf \} \\
% & & \vee x_{[i,i+1]} \in \{ \gdos  : \gamma_{1}(x)_{i} \in \{ \cero , \uno  \} \} 
%\end{array}
%\end{displaymath}
In other words, on the first coordinate $T$ acts exactly as $T_1$; on the second coordinate, an arrow advances to the left if and only if the first coordinate is a $\vacio$ or a $\dos$. In case two arrows overlap they superimpose each other (i.e. the CA is not conservative).

In the introduction, we said this construction was a ``local" skew product. On a skew product the phase space is the product, one coordinate acts normally and the second acts only if the first coordinate is in a certain position. The main difference is that here the shift only acts locally. 

Though the next table is not needed for the proofs we provide it in case it is of assistance to the reader.

\begin{displaymath}
\begin{array}{ccc}
Tx_{i}& = & \displaystyle\left\{ \begin{array}{ccl}
\vacio & if & x_{i}\in \{ \vacio , \flechai \}  \wedge x_{i+1} \in A\setminus \{ \flechai , \dosf \} ;\\
%      &    & (x_{i-1} \in A \wedge x_{i}=\flechai \wedge x_{i+1}\in A\setminus \{ \flechai , \dosf \});\\

\flechai & if & x_{i}\in \{ \vacio , \flechai \} \wedge x_{i+1}\in \{ \flechai, \dosf \} ;\\
%         &    & (x_{i-1}\in \{ \flechai, \cerof,\unof, \dosf \} \wedge x_{i}=\flechai );\\

\cero & if & (x_{i}\in \{ \dos , \dosf \}  \wedge x_{i+1}\in \{ \vacio , \dos \} ) \vee ;\\
%     &    & (x_{i-1}\in \{ \vacio, \cero, \uno, \dos \} \wedge x_{i}=\dosf \wedge x_{i+1}\in \{ \vacio,                 			\dos , \flechai , \dosf \} ) \vee ; \\    
     &    & (x_{i}=\cero \wedge x_{i+1}\in A\setminus \{ \vacio  , \flechai,\dos , \dosf \} ) ;\\

\uno & if & (x_{i}= \cero \wedge x_{i+1}\in \{ \vacio , \dos \} ) \vee ;\\
     &    & (x_{i}=\uno \wedge x_{i+1}\in A\setminus \{ \vacio  , \flechai,\dos , \dosf \} ) ;\\

\dos & if & (x_{i}=\uno \wedge x_{i+1}\in \{ \vacio , \dos \} ) \vee \\
      &    & (x_{i}\in \{ \dos , \dosf \} \wedge x_{i+1}\in A\setminus \{ \vacio  , \flechai,\dos , \dosf \} );\\
%      &    & (x_{i-1}\in \{ \vacio, \cero, \uno, \dos \} \wedge x_{i}=\dosf \wedge x_{i+1}\in A\setminus
%             \{ \vacio  , \flechai,\dos , \dosf \} ) ;\\

\cerof & if & (x_{i}\in \{ \dos ,\dosf \}  \wedge x_{i+1}\in \{ \flechai,\dosf \} ) \vee ; \\
      &    & (x_{i}=\cerof\wedge x_{i+1}\in A\setminus \{ \vacio ,\flechai,\dos ,\dosf\} ) ;\\
%     &    & x_{i}=\dosf \wedge x_{i+1}\in \{\vacio , 			 \dos , \flechai , \dosf \} );\\

\unof & if & (x_{i}=\cero \wedge x_{i+1}\in  \{ \flechai, \dosf \} ) \vee  \\
      &    & (x_{i}=\cerof \wedge x_{i+1}\in \{ \vacio ,\flechai,\dos ,\dosf \} ) \vee \\ 
      &    & (x_{i}=\unof \wedge x_{i+1}\in A\setminus \{ \vacio  , \flechai,\dos , \dosf \});\\

\dosf & if & (x_{i}= \uno \wedge x_{i+1}\in  \{ \flechai, \dosf \} ) \vee  \\ 
       &    & (x_{i}=\unof \wedge x_{i+1} \in \{ \vacio  , \flechai,\dos , \dosf \} ). \\
%       &    & (x_{i-1}\in\{ \flechai, \cerof, \unof, \dosf \} \wedge x_{i}=\dosf \wedge x_{i+1}\in A					  \setminus \{ \vacio  , \flechai,\dos , \dosf \}). 

\end{array}\right.

\end{array}
\end{displaymath}

%Basically, there two rules for the movement of \flechai :

%\begin{enumerate}
%\item If there is not \flechai \ to the left of another \flechai;
%\item If there is \vacio \ or \dos \ to the left of \flechai.
%\end{enumerate}

When we write $x:=\infinito\vacio. \cero \ \flechai^{\infty}$, we mean the point $x$ such that $x_0=\cero$, $x_{-i}=\vacio$ and $x_i=\flechai$ for every $i\neq 0$.
\pagebreak
\begin{example}\label{4puertas}
Let $x\in A^{\ZZ}$ such that $x=\infinito \vacio$.\vacio \cero \cero \cero \cero $\flechai^{\infty}$. We have that

	\begin{displaymath}
	\begin{array}{cccccccccccccccccccc}
T^{0}x_{[0,6]} & = & \vacio &\cero &\cero &\cero &\cerof &\flechai & &T^{14}x_{[0,6]} & = &\vacio &\unof &\cero &\unof  &\dosf & \flechai &\\
T^{1}x_{[0,6]} & = &\vacio &	\cero   &\cero   &\cero    &\unof  &\flechai & & T^{15}x_{[0,6]} & = &\vacio &	\unof   &\cero   &\dosf   &\cerof   &\flechai\\
T^{2}x_{[0,6]} & = &\vacio &	\cero   &\cero   &\cero    &\dosf &\flechai  & & T^{16}x_{[0,6]} & = &\vacio &	\unof   &\unof   &\dos   &\unof   &\flechai\\
T^{3}x_{[0,6]} & = &\vacio &	\cero   &\cero   &\unof &\cerof &\flechai & & T^{17}x_{[0,6]} & = &\vacio & \unof &\dosf  &\dos  &\dosf &\flechai \\
T^{4}x_{[0,6]} & = &\vacio &	\cero   &\cero   &\unof    &\unof  &\flechai  & & T^{18}x_{[0,6]} & = &\vacio & \dosf  &\cero   &\cerof    &\cerof    &\flechai\\
T^{5}x_{[0,6]} & = &\vacio &	\cero   &\cero   &\unof    &\dosf  &\flechai & & T^{19}x_{[0,6]} & = &\flechai &	\dos  &\cero   &\cerof    &\unof    &\flechai\\
T^{6}x_{[0,6]} & = &\vacio &	\cero   &\cero   &\dosf   &\cerof  &\flechai & & T^{20}x_{[0,6]} & = &\vacio &	\dos  &\cero   &\cerof    &\dosf  &\flechai\\
T^{7}x_{[0,6]} & = &\vacio &	\cero   &\unof   &\dos   &\unof  &\flechai  & & T^{21}x_{[0,6]} & = &\vacio &	\dos  &\cero   &\unof    &\cerof    &\flechai\\
T^{8}x_{[0,6]} & = &\vacio & \cero  &\dosf &\dos   &\dosf  & \flechai & & T^{22}x_{[0,6]} & = &\vacio &	\dos  &\cero   &\unof    &\unof     &\flechai\\
T^{9}x_{[0,6]} & = &\vacio & \unof   &\cero   &\cerof    &\cerof   &\flechai  & & T^{23}x_{[0,6]} & = &\vacio &	\dos  &\cero   &\unof    &\dosf   &\flechai \\
T^{10}x_{[0,6]} & = &\vacio &	\unof   &\cero   &\cerof    &\unof &\flechai  & & T^{24}x_{[0,6]} & = &\vacio &	\dos  &\cero   &\dosf   &\cerof     &\flechai \\
T^{11}x_{[0,6]} & = &\vacio &	\unof   &\cero   &\cerof    &\dosf &\flechai & & T^{25}x_{[0,6]} & = &\vacio &	\dos  &\unof   &\dos   &\unof  &\flechai \\
T^{12}x_{[0,6]} & = &\vacio &	\unof   &\cero   &\unof    &\cerof & \flechai  & & T^{26}x_{[0,6]} & = &\vacio &	\dos  &\dosf  &\dos   &\dosf  &\flechai \\
T^{13}x_{[0,6]} & = &\vacio &	\unof   &\cero   &\unof    &\unof  &\flechai & & T^{27}x_{[0,6]} & = &  \vacio & \cerof   &\cero   &\cerof    &\cerof & \flechai .\\
 
	\end{array}
	\end{displaymath}

\end{example}

\begin{example}\label{4puertas}
Let $x\in A^{\ZZ}$ such that $x=\infinito \vacio$.\vacio \cero \vacio \cero \cero $\flechai^{\infty}$. We have that

	\begin{displaymath}
	\begin{array}{cccccccccccccccccccc}
T^{0}x_{[0,6]} & = & \vacio &\cero &\vacio &\cero &\cero &\flechai & &T^{14}x_{[0,6]} & = &\vacio &\dos &\vacio &\unof  &\dosf & \flechai &\\
T^{1}x_{[0,6]} & = &\vacio &	\uno   &\vacio   &\cero    &\unof  &\flechai & & T^{15}x_{[0,6]} & = &\vacio &	\cero   &\vacio   &\dosf   &\cerof   &\flechai\\
T^{2}x_{[0,6]} & = &\vacio &	\dos   &\vacio  &\cero    &\dosf &\flechai  & & T^{16}x_{[0,6]} & = &\vacio &	\uno   &\flechai   &\dos   &\unof   &\flechai\\
T^{3}x_{[0,6]} & = &\vacio &	\cero   &\vacio   &\unof &\cerof &\flechai & & T^{17}x_{[0,6]} & = &\vacio & \dosf &\vacio  &\dos  &\dosf &\flechai \\
T^{4}x_{[0,6]} & = &\vacio &	\uno   &\vacio   &\unof    &\unof  &\flechai  & & T^{18}x_{[0,6]} & = &\flechai & \cero  &\vacio   &\cerof    &\cerof    &\flechai\\
T^{5}x_{[0,6]} & = &\vacio &	\dos   &\vacio   &\unof    &\dosf  &\flechai & & T^{19}x_{[0,6]} & = &\vacio &	\uno  &\vacio   &\cerof    &\unof    &\flechai\\
T^{6}x_{[0,6]} & = &\vacio &	\cero   &\vacio   &\dosf   &\cerof  &\flechai & & T^{20}x_{[0,6]} & = &\vacio &	\dos  &\vacio   &\cerof    &\dosf  &\flechai\\
T^{7}x_{[0,6]} & = &\vacio &	\uno   &\flechai   &\dos   &\unof  &\flechai  & & T^{21}x_{[0,6]} & = &\vacio &	\cero  &\vacio   &\unof    &\cerof    &\flechai\\
T^{8}x_{[0,6]} & = &\vacio & \dosf  &\vacio &\dos   &\dosf  & \flechai & & T^{22}x_{[0,6]} & = &\vacio &\uno  &\vacio   &\unof    &\unof     &\flechai\\
T^{9}x_{[0,6]} & = &\flechai & \cero   &\vacio   &\cerof    &\cerof   &\flechai  & & T^{23}x_{[0,6]} & = &\vacio &	\dos  &\vacio   &\unof    &\dosf   &\flechai \\
T^{10}x_{[0,6]} & = &\vacio &	\uno   &\vacio   &\cerof    &\unof &\flechai  & & T^{24}x_{[0,6]} & = &\vacio &	\cero  &\vacio   &\dosf   &\cerof     &\flechai \\
T^{11}x_{[0,6]} & = &\vacio &	\dos   &\vacio   &\cerof    &\dosf &\flechai & & T^{25}x_{[0,6]} & = &\vacio &	\uno  &\flechai   &\dos   &\unof  &\flechai \\
T^{12}x_{[0,6]} & = &\vacio &	\cero   &\vacio   &\unof    &\cerof & \flechai  & & T^{26}x_{[0,6]} & = &\vacio &	\dosf  &\vacio  &\dos   &\dosf  &\flechai \\
T^{13}x_{[0,6]} & = &\vacio &	\uno   &\vacio   &\unof    &\unof  &\flechai & & T^{27}x_{[0,6]} & = &  \flechai & \cero   &\vacio   &\cerof    &\cerof & \flechai \\
 
	\end{array}
	\end{displaymath}

\end{example}

%The following lemma can be checked by simple use of the rules of $T$. 
%\begin{lemma}\label{propo1puerta1fantasma}

%Let $x\in A^{\ZZ}$ such that $x_{[0,1]}=\vacio \ \cero$. For all $k\geq 1$ we have that $T^{k}x_{0}\in \{\vacio ,\flechai \}$. Furthermore, if $\gamma_1(T^{n}x_{1})\neq\dos$ then $T^{n+1}x_0=\vacio$. 

%Let $x\in A^{\ZZ}$ such that $x_{0}=$\cero, $x_{1}= \flechai$ and $x_{i}=\vacio$ for all $i\in \ZZ \setminus \{ 0,1 \}$. Then $T^{3}x_{-1}=$\flechai , $T^{3}x_{0}=$\cero, and $T^{3}x_{i}=\vacio$ for all $i\in \ZZ \setminus \{ -1, 0 \}$.

%\end{lemma}

%\begin{proof}This can be checked step by step. Let $x\in A^{\ZZ}$ be as in the proposition. We have that\begin{itemize}\item $T^{1}x_{0}=$\unof \ and $T^{1}x_{i}=\vacio$ for all $i\neq 0$;\item $T^{2}x_{0}=$\dosf \ and $T^{1}x_{i}=\vacio$ for all $i\neq 0$;\item $T^{3}x_{-1}=$\flechai, $T^{3}x_{0}=$\cero, and $T^{3}x_{i}=\vacio$ for all $i\in \ZZ \setminus \{ -1, 0 \}$.\end{itemize}\end{proof}

\begin{lemma}\label{pared k=1} Let $x\in A^{\ZZ}$ such that $x_{[0,2]}=\vacio \ \textnormal{\cero} \ \vacio $. Then $T^{n}x_{0}=\vacio$ for all $n\neq 3k$, where $k \geq 1$.
\end{lemma}

\begin{proof}
From Lemma \ref{uno l=0} we have that $\gamma_{1}(T^{n} x_{1})=\dos$ if and only if $n=3k-1$, where $k\geq 1$. Therefore, $T^{n}x_{0}=\vacio$ for all $n\neq 3k$, where $k \geq 1$. 
\end{proof}

\begin{lemma}\label{pared k=l}
Let $k,l\geq 0$, and $x\in A^{\ZZ}$ such that $x_{[0,2^{l}+1]} = \vacio \ \textnormal{\cero}^{2^{l}} \vacio$. If $n\neq k(3^{l+1})+2(3^{l})+1$ then $T^{n}x_{0}=\vacio$.
%and $T^{j}x_{-1} =\vacio$, for all $k\geq 0$ and for all $7+9k<j< 7+9(k+1)$. 
\end{lemma}

\begin{proof}
We will use the following two properties that are easily checked for any $y\in A^{\ZZ}$.  
\begin{itemize}
\item  If $\gamma_1( y_1)\notin \{\vacio,\dos\}$ and $\gamma_{2}(y_{0})=\vacio$ then $\gamma_2 (Ty_0)=\vacio$.

\item If  $\gamma_1( y_2)\notin \{\vacio,\dos\}$ and $\gamma_1( y_1)=\dos $ then $\gamma_2 (T^2y_0)=\vacio$.
\end{itemize}
Now by Proposition \ref{lemaNpuertas} we have that $(\gamma_1(x))_1$ has period $3^{l+1}$ (for $T_1$) and

\begin{itemize}
      \item $\gamma_{1}(T^{i}x_{1})=\cero$ for all $0\leq i <3^{l}$;
      \item $\gamma_{1}(T^{i}x_{1})=\uno$ for all $3^{l}\leq i <2(3^{l})$;
      \item $\gamma_{1}(T^{i}x_{1})=\dos$ for all $2(3^{l})\leq i <3^{l+1}$.
\end{itemize}
This implies that $\gamma_{1}(T_{1}^{i}x_{2})\neq \dos$ for all $2(3^{l})\leq i <3^{l+1}-1$ (otherwise value on coordinate $1$ would change).

Using the previous properties we conclude that $T^{i}x_{0}=\vacio$ for all $2(3^{l})+2\leq i \leq 3^{l+1}$. Actually, using the periodicity of $\gamma_{1}(x_{[0,2^{l}+1]})$ (Proposition \ref{lemaNpuertas}) we conclude that $T^{n}x_{0}=\vacio$ for all $n\neq k(3^{l+1})+2(3^{l})+1$, for some $k\geq 0$.
\end{proof}

\begin{lemma}\label{palabrasinfantasmas}
	Let $x\in A^{\ZZ}$ and $j\in \ZZ$ such that $\gamma_{2}(x_{i})= \vacio $ for all  $i\geq j$ and $\gamma_{1}(x_{j})=\vacio$. 
	For every $k\geq j$ there exists $N>0$ such that $\gamma_{2}(T^n x_k)=\vacio$ for every $n\geq N$.
\end{lemma}

\begin{proof}
Using Proposition \ref{Siempre hay una salida} we can see that every arrow on a position $k\geq j$ will eventually move to the left (since they move when $\gamma_1$ is $\vacio$ or $\dos$). For every $k\geq j$ we have that $\gamma_{2}(x_n)=\flechai$ for only finitely many $n>k$. Since arrows only move to the left one can conclude the result. 
\end{proof}

The following statement will be used to show that $(A^{\ZZ},T)$ is not almost equicontinuous. 

\begin{proposition}\label{ningunpunto-m-e-p}
Let $m>0$ and $w\in A^{2m+1}$. There exists $x,y\in A^{\ZZ}$ with $x_{[-m,m]}=w=y_{[-m,m]}$ such that $T^{n}x_{0}\neq T^{n}y_{0}$ for some $n>0$.
\end{proposition}

\begin{proof}
Let $x,y$ as in the hypothesis and with $x_{i}=\vacio$ and $y_{i} =\flechai$, for every $i\in \ZZ$ such that $|i|>m$.  By Lemma \ref{palabrasinfantasmas} we have that there exists $N>0$ such that for all $n\geq N$ we have that $\gamma_{2}(T^{n}x_{i})=\vacio$ for all $i\in [-m,m]$. Using that $y_{i}=\flechai$ for all $i> m$ and Proposition \ref{Siempre hay una salida} we obtain that $\{n\in \NN:\gamma_2(T^ny_0)=\flechai\}$ is infinite. Hence, there exists $N'\in \{n\in \NN:\gamma_2(T^ny_0)=\flechai\}$ such that $N'>N$ and $T^{N'}x_{0}\neq T^{N'}y_{0}$.

%We set $x_i=\vacio$ and $y_i=\flechai$, for every $i>m$. 	
%Let $l_x=$sup$\{i\in\ZZ :\gamma_{2}(x_{i})=\flechai \}$. Note that $l_{Tx}\in \{l_x,l_x-1\}$. Furthermore, by the rules of $T$ and Proposition \ref{Siempre hay una salida}, there exists $k\in \NN$ such that $l_{T^kx}< l_x$. This implies that $\gamma_2(T^nx_0)=\vacio$  for sufficiently large $n$. On the other hand, an application of Proposition \ref{Siempre hay una salida} also gives us that $\{n\in \NN:\gamma_2(T^ny_0)=\flechai\}$ is infinite. 

\end{proof}
%The proof of the following lemma is similar. 
%\begin{lemma}\label{palabrasinfantasmas}
%	Let $x\in A^{\ZZ}$ and $j\in \ZZ$ such that $\gamma_{2}(x_{i})= \vacio $ for all  $i\geq j$. 
%	For every $k\in \ZZ$ there exists $N>0$ such that $\gamma_{2}(T^n x_k)=\vacio$ for every $n\geq N$.
%\end{lemma}

In the following lemma we will use the sets $S_{\{ -i,i \} }(x,m)$ from Definition \ref{S_J}.  

\begin{remark}
	\label{rem:characS}
	Let $m>0$ and $x\in A^{\ZZ}$ so that $\gamma_2(x_i)=\vacio$ for all $i\in \ZZ$ and $\gamma_{1}( x_{m})=\vacio$. Then $$S_{\{ j\} }(x,m)=\{ i\in \NN : \exists \ y\in B_{m}(x), \ \gamma_{2}(T^{i}y_{j})=\flechai \}.
	$$ for all $|j|\leq m$.
\end{remark}
\begin{lemma}\label{densidadpuertas}

Let $l\geq 0$ and $x\in A^{\ZZ}$ such that $x_{[0,2^{l}+1]} = \vacio \ \textnormal{\cero}^{2^{l}} \vacio$. Then

$$\overline{D} (S_{\{0\}}(x,2^{l}+1)) \leq \frac{1}{3^{l+1}} .$$ 

\end{lemma}

\begin{proof}
%By Lemma  \ref{pared k=l}, we have that 
%$$S_{\{0\}}(x,2^{l}+1)\subseteq \{ i\in \ZZ_{0\geq} : i=3^{l+1}k+2(3^{l})+1 \ \forall \ k\geq 0 \}.$$ 
%Hence,
%$$\overline{D}(S_{\{0\}}(x,2^{l}+1))\leq \overline{D}( \{ i\in \ZZ_{0\geq} : i=3^{l+1}k+2(3^{l})+1 \ \forall \ k\geq 0 \} )$$ 
%$$ = \limsup_{k\to \infty}\frac{k}{3^{l+1}k+2(3^{l})+1)}=\frac{1}{3^{l+1}} .$$

%Let $y,z\in A^{\ZZ}$ with $y_{[0,2^{l}+1]}=x_{[0,2^{l}+1]}=z_{[0,2^{l}+1]}$, $y_{2^{l}+1+i}=\vacio =y_{-i}$, $z_{2^{l}+1+i}=\flechai$ and $z_{-i}=\vacio$, for every $i\geq 1$. Note that $x$ is a fixed point. 
By Lemma \ref{pared k=l} we have that 
$$S_{\{0\}}(x,2^{l}+1) \subseteq \{ i\in \NN : i=3^{l+1}k+2(3^{l})+1 \ \forall \ k\geq 0 \}.$$
Therefore,
$$\overline{D}(S_{\{0\}}(x,2^{l}+1)) \leq \overline{D}( \{ i\in \NN : i=3^{l+1}k+2(3^{l})+1 \ \forall \ k\geq 0 \} )$$ 
$$ = \limsup_{k\to \infty}\frac{k+1}{3^{l+1}k+2(3^{l})+1}=\frac{1}{3^{l+1}} .$$
\end{proof}

\begin{lemma}\label{densidadededosspuertas}
Let $l\geq 0$, $m =2^{l}+2^{l-1}+2$ and $x\in A^{\ZZ}$ with $x_{[0,m]}=\vacio \ \textnormal{\cero}^{2^{l-1}} \vacio \ \textnormal{\cero}^{2^{l}} \ \vacio$
%$x:= \infinito\vacio .w \vacio^{\infty}$, and . 
Then
\begin{displaymath}
\begin{array}{rcl}

\overline{D}(S_{\{ i\} }(x,m)) \leq  (2(3^{l-1})+1)\overline{D}(S_{\{ 2^{l-1}+1\} }(x,m)) \text{, and}
\\
\\
\overline{D}(S_{\{0\}}(x,m)) = 
%\overline{D}(S_{\{ -i\} }(x,m)) =
\overline{D}(S_{\{ 2^{l-1}+1\} }(x,m))
%\overline{D}(S_{\{ -i\} }(x,m)) &= & \overline{D}(S_{\{ 2^{l-1}+1\} }(x,m))\text{, and }  

\end{array}
\end{displaymath}
for all $ i\in [1, 2^{l-1}]$.
\end{lemma}

\begin{proof}
	
 By Lemma \ref{lemaNpuertas} we have that
\begin{itemize}
      \item $\gamma_{1}(T^{n}x_{1})=\cero$ for all $0\leq n <3^{l-1}$;
      \item $\gamma_{1}(T^{n}x_{1})=\uno$ for all $3^{l-1}\leq n <2(3^{l-1})$;
      \item $\gamma_{1}(T^{n}x_{1})=\dos$ for all $2(3^{l-1})\leq n <3^{l}$.
\end{itemize}
Let $ 1\leq i\leq 2^{l-1}$. Any arrow on this region will not move for at most $2(3^{l-1})$ iterations, that is
for any $y\in B_{m}(x)$ we have that $\gamma_{2}(T^{n}y_{i})=\flechai$ for at most $2(3^{l-1})$ consecutive $n$.
% Now $\gamma_{2}(T^{n}x_{1})=\vacio$ for all $2(3^{l-1})<n<3^{l}$, because $\gamma_{1}(T^{n}x_{2})\neq \dos$ for all $2(3^{l-1})\leq n < 3^{l}-1$. Since
% $$T^{2(3^{l})+1}y_{2^{l-1}+1}\neq T^{2(3^{l})+1}z_{2^{l-1}+1} \text{ and } $$
% $$T^{n}y_{2^{l-1}+1} = T^{n}z_{2^{l-1}+1}$$ 
%for all $n\in [0,3^{l+1}] \setminus \{ 2(3^{l})+1\}$,
%By these and since 
%\begin{itemize}
%\item $T^{k(3^{l+1})+2(3^{l})+1}y_{2^{l-1}+1}\neq T^{k(3^{l+1})+2(3^{l})+1}z_{2^{l-1}+1}$ and $T^{n}y_{2^{l-1}+1} = T^{n}z_{2^{l-1}+1}$ %for all $n\neq k(3^{l+1})+2(3^{l})+1$, and
%\item  $T^{k(3^{l+1})+2(3^{l})+1+2(3^{l-1})}y_{0}\neq T^{k(3^{l+1})+2(3^{l})+1+2(3^{l-1})}z_{0}$ and $T^{n}y_{0} = T^{n}z_{0}$, for all %$n \neq k(3^{l+1})+2(3^{l})+1+2(3^{l-1})$,
%\end{itemize}
Hence, by Remark \ref{rem:characS} we can conclude that
%$$\sharp (+S_{i}(x,m)\cap [0,3^{l+1}])  \leq   (2(3^{l-1})+1)\sharp (+S_{2^{l-1}+1}(x,m)\cap [0,3^{l+1}]).$$
%Again, by Lemma \ref{lemaNpuertas} we have that if   
%$$\gamma_{1}(T^{n}x_{[i,2^{l-1}]})=\gamma_{1}(x_{[i,2^{l-1}]}),$$
%for any $1\leq i\leq 2^{l-1}$, then $n=3^{h}$ where $1\leq h\leq l$. Hence, we have that
$$ \sharp (S_{\{ i\} }(x,m)\cap [0,(k+1)3^{l+1}]) \leq  (2(3^{j})+1)\sharp (S_{\{ 2^{l-1}+1\} }(x,m)\cap [0,(k+1)3^{l+1}]),$$
for all $k\geq 0$.
Therefore,
$$\overline{D}(S_{\{ i\} }(x,m))\leq (2(3^{l-1})+1)\overline{D}(S_{\{ 2^{l-1}+1\} }(x,m)) .$$
%Then,
%\begin{displaymath}
%\begin{array}{rcl}
%\sharp (+S_{i}(x,m)\cap [0,(k+1)3^{l+1}]) \overline{D}(+S_{2^{j}+1}(x,m))& \leq & k (2(3^{j})+1) \overline{D}(+S_{2^{j}+1}(x,m)).
%\end{array}
%\end{displaymath}  
%So, we have that
%\begin{displaymath}
%\begin{array}{rcl}
%\frac{(2(3^{j})+1)\overline{D}(+S_{2^{j}+1}(x,m))}{3^{l+1}}  %& =& \lim_{k\to \infty}\frac{k (2(3^{j})+1)\overline{D}(+S_{2^{j}+1}(x,m))}{(k+1)3^{l+1}} \\
%& \geq & \lim_{k\to \infty}\frac{\sharp (+S_{i}(x,m)\cap [0,(k+1)3^{l+1}]) \overline{D}(+S_{2^{j}+1}(x,m))}{(k+1)3^{l+1}} .\\
%\end{array}
%\end{displaymath}
%Now, observe that for any $n\geq 0$ there exists a $k\geq 1$ such that $k(3^{l+1})+2(3^{l-1})+1 \leq  n\leq (k+1)(3^{l+1})+2(3^{l-1})+1$. So, for such $n$ and $k$ we have that
%$$\frac{\sharp (+S_{i}(x,m)\cap [0,n])}{n}\leq \frac{(2(3^{l-1})+1)\sharp (+S_{2^{l-1}+1}(x,m)\cap [0,(k+1)3^{l+1}])}{(k+1)(3^{l+1})+2(3^{l-1})+1}$$
%Therefore, 

%$$\overline{D}(+S_{i}(x,m))\leq (2(3^{l-1})+1)\overline{D}(+S_{2^{l-1}+1}(x,m)) ,$$
%for all $ 1\leq i\leq 2^{l-1}$ and all $m\geq 2^{l-1}+2^{l}+1$.

Now, for $y\in B_{m}(x)$ one can check that $\gamma_2(T^nx_{2^{l-1}+1})=\flechai$ if and only if $\gamma_2(T^{n+2(3^{l-1})}x_{0})=\flechai$ (see Proposition \ref{lemaNpuertas} and Lemma \ref{pared k=1}). Using this and Remark \ref{rem:characS}, we obtain that
 $$S_{\{0\}}(x,m)=S_{\{2^{l-1}+1\}}(x,m) +2(3^{l-1}).$$ 
 Therefore,
$$\overline{D} (S_{\{0\}}(x,m)) = \overline{D}(S_{\{ 2^{l-1}+1\} }(x,m)).$$

%There exist $y,z\in A^{\ZZ}$ such that 
%\begin{itemize}
%\item $y_{[0,2^{l-1}+2^{l}+2]}  = x_{[0,2^{l-1}+2^{l}+2]}=z_{[0,2^{l-1}+2^{l}+2]},$ 
%\item $y_{2^{l-1}+2^{l}+2+i}=\vacio =y_{-i}$, and 
%\item $z_{2^{l-1}+2^{l}+2+i}=\flechai$ and $z_{-i}=\vacio$,
%\end{itemize}
%for all $i\geq 1$. By this and Lemma \ref{pared k=l}, we have that  $T^{2(3^{l})+1}y_{2^{l-1}+1}\neq T^{2(3^{l})+1}z_{2^{l-1}+1}$ and $T^{n}y_{2^{l-1}+1} = T^{n}z_{2^{l-1}+1}$ for all $n\in [0,3^{l+1})\setminus \{ 2(3^{l})+1\} $. By Lemma \ref{lemaNpuertas} we have that $\gamma_{1}(T^{n}x_{[1,2^{l-1}]})$ has period $3^{l}$. Then
%\begin{itemize}
%\item $T^{2(3^{l})+1+2(3^{l-1})}y_{0}\neq T^{2(3^{l})+1+2(3^{l-1})}z_{0}$ and 
%\item $T^{n+2(3^{l-1})}y_{0} = T^{n+2(3^{l-1})}z_{0}$, 
%\end{itemize}
%for all $n\in [0,3^{l+1}+2(3^{l-1}))\setminus \{ 2(3^{l})+1+2(3^{l-1})\} $. Again by Lemma \ref{lemaNpuertas} we have that $\gamma_{1}(T^{n}x_{[2^{l-1}+2,2^{l-1}+2^{l}+1]})$ has period $3^{l+1}$. So, we have that
%$$S_{\{0\}}(x,m)=\{ i+2(3^{l-1}) : i\in  +S_{2^{l-1}+1}(x,m)  \} .$$

\end{proof}

\begin{proposition}\label{punto-d-m-e-propo}
Let $k>0$, $w\in A^{k}$ and 
$$x:= \infinito\vacio .w \vacio \ \textnormal{\cero \vacio  \cero}^{2}  \textnormal{\vacio \cero}^{2^{2}} \cdots \textnormal{\cero}^{2^{n}} \cdots.$$ 
We have that $x$ is a diam-mean equicontinuity point.

\end{proposition}
\begin{proof}

	We will prove that $x$ is a diam-mean equicontinuity point with the use of Proposition \ref{D-M-E-P}. Let $m\geq 0$. First notice that, without loss of generality, we may assume that $\gamma_{2}(w_{i})=\vacio$ for every $1\leq i \leq k$ (from Lemma \ref{palabrasinfantasmas} there exists $M>0$ such that $\gamma_{2}( T^{M}x_{i}) = \vacio$ for all $0\leq i <k$). 
	Let $l>0$ so that $k<2^l$ and
$$ \frac{2(3^{l-1}+1)}{3^{l+1}}\leq \frac{1}{2^{m+2}}.$$
Let $k\leq j\leq  k+l+\sum_{i=0}^{l-1} 2^{i}$. By applying  Lemma \ref{densidadededosspuertas} recurrently (and using that $k<2^l$), we have that  
\begin{equation}\label{EqPropDMP}
\overline{D}(S_{\{j\} }(x,k+l+\sum_{i=0}^{l} 2^{i})) \leq  2(3^{l-1}+1)\overline{D}(S_{\{ k+l+\sum_{i=0}^{l-1} 2^{i}\} }(x,k+l+\sum_{i=0}^{l} 2^{i})).
\end{equation}
Recall, by Remark \ref{rem}, that $\gamma_{1}(T^{n}x_{[0,k]})=\gamma_{1}(T^{n}y_{[0,k]})$ for all $y\in B_{k+l+\sum_{i=0}^{l}2^{i}}(x)$ and all $n\geq 0$. Hence, since $k<2^l$ we can conclude  all $j\leq  k+l+\sum_{i=0}^{l-1} 2^{i}$ satisfies Equation \eqref{EqPropDMP}.
%Let $0\leq j <k$. Let us suppose that 
%$$\overline{D}(S_{\{ j \} }(x,k+l+\sum_{i=0}^{l}2^{i}) ) > 2(3^{l-1}+1)\overline{D}(S_{\{ k+l+\sum_{i=0}^{l-1}2^{i} \} }(x,k+l+\sum_{i=0}^{l}2^{i}) ).$$
%By Lemma \ref{densidadpuertas} we have that
%$$\overline{D}(S_{\{ j \} }(x,k+l+\sum_{i=0}^{l}2^{i}) ) > \frac{2(3^{l-1}+1)}{3^{l+1}}.$$ 
%This implies that $\gamma_{2}(T^{n}x_{i})=\flechai$ for at most $2(3^{l-1})$ consecutive $n$.
By Lemma \ref{densidadpuertas} and the choice of $l$ we obtain that
\begin{displaymath}
	\begin{array}{rcl}
		\overline{D}(S_{\{j\} }(x,k+l+\sum_{i=0}^{l} 2^{i}))  &\leq &
		\frac{2(3^{l-1}+1)}{3^{l+1}} \\
		&\leq &
		\frac{1}{2^{m+2}},
	\end{array}
\end{displaymath}
for all $j\leq  k+l+\sum_{i=0}^{l-1} 2^{i}$. 

It is not hard to check that that$$\overline{D}(S_{\{-j\} }(x,k+l+\sum_{i=0}^{l} 2^{i}))\leq \overline{D}(S_{\{ 0\} }(x,k+l+\sum_{i=0}^{l} 2^{i}))$$for every $j>0$.

%Indeed for $y\in B_{2^{-m}}(x)$ and $0\leq i \leq 2^{l-1}$ we have that $T^{n+1}y_{-(i+1)}=\flechai$ if and only if $T^{n}y_{-i}=\flechai$. This implies that$$S_{\{ -i\} }(x,m)=S_{\{-1\}}(x,m)+i.$$With this we obtain the property.

Therefore, by Proposition \ref{D-M-E-P}, we have that $x$ is a diam-mean equicontinuity point.
%|w|+m+1+(\sum_{l=0}^{m} 2^{l})

\end{proof}

\begin{theorem}
	\label{thm:main}
The CA $(A^{\ZZ},T)$ is almost diam-mean equicontinuous but not almost equicontinuous. 
\end{theorem}
\begin{proof}
From Proposition \ref{punto-d-m-e-propo} we have $EQ$ is dense. Hence, by Proposition \ref{pro sigma}, $EQ$ is residual. So, $(A^{\ZZ},T)$ is almost diam-mean equicontinuous. By Proposition \ref{ningunpunto-m-e-p} there are no equicontinuity points. Therefore, $(A^{\ZZ},T)$ is not almost equicontinuous. 
\end{proof}
%\begin{proposition}
%The CA $(A^{\ZZ},T)$ is not mean equicontinuous.
%\end{proposition}

%Let $x\in A^{\ZZ}$ such that $x_{i}=\cerof$ for all $i\in \ZZ$. Observe that $T^{n}x=x$ for all $n\geq 0$. For all $m>0$ there exists $y\in B_{2^{-m}}(x)$ such that $y_{i}=\vacio$ for all $|i|\geq m$. Hence, by Lemma \ref{palabrasinfantasmas}, there exists $N>0$ such that for all $n>N$ we have that $T^{n}x_{0}=T^{n}y_{0}$. Thus, $x$ is not a mean equicontinuous point. Therefore, $(A^{\ZZ}, T)$ is not mean equicontinuous. 

%\begin{proposition}
%$(A^{\ZZ},T)$ is  not cofinitely sensitive but it is syndetically sensitive.
%\end{proposition}

%\begin{proof}
%Let $w\in A^{+}$ such that $w=\vacio \ $\cero. From Proposition \ref{propo1puerta1fantasma} we have that %$N_{T}([w],1)$ is not a cofinite set.

%Let $w\in A^{+}$. By Lemma \ref{palabrasinfantasmas} we can assume that $w_{i}\notin \{\flechai, \cerof, \unof, \dosf\}$ for all $1\leq i\leq |w|$. So

%\end{proof}

\section{Diam-mean sensitivity}
A TDS, $(X,T)$, is \textbf{sensitive} if there exists $\e >0$ such that for every non-empty open set $U\subseteq X$ there exist $x,y\in U$ and $n>0$ 
such that $$d(T^{n}x,T^{n}y)> \e.$$

Diam-mean sensitivity was introduced in \cite{garciajagerye}.
\begin{definition}
Let $(X,T)$ be a topological dynamical system. We say that $(X,T)$ is \textbf{diam-mean sensitive} if there exists $\e >0$ such that for every open set $U$ we have 
$$\limsup_{n\to \infty}\frac{\sum_{i=1}^{n}\operatorname{diam}(T^{i}U)}{n} > \e .$$
\end{definition}

We say a TDS is \textbf{transitive} if for every pair of non-empty open sets $U$ and $V$ there exists $n>0$ such that $T^{-n}U\cap V\neq \emptyset$.
% $\n =X\times X$, where $\n = \{ (x,y)\in X\times X : \forall \e ,\d >0, \ \exists n>0, \ \exists z\in B_{\d }(x), \ d(T^{n}(x),y)<\e  \}$.

\begin{theorem}\label{Teo:Dicotomianormaldiammean}
	Let $(X,T)$ be a transitive TDS. 
	\begin{itemize}
		\item $(X,T)$ is either almost equicontinuous or sensitive \cite{kurka1997languages}; and
		\item $(X,T)$ is either almost diam-mean equicontinuous or diam-mean sensitive \cite{weakforms}.
		\end{itemize}
\end{theorem}
For other related dichotomies see \cite{huang2018analogues}. 

The first part of the previous theorem holds for non-necessarily transitive CA. 
\begin{theorem}
	[\cite{kurka1997languages}]
	\label{thm:dicotomiaCA}
	Let $(X,T)$ be a CA. Then $(X,T)$ is either almost equicontinuous or sensitive.
\end{theorem}

%We have that if a TDS is diam-mean sensitive then it does not have diam-mean equicontinuous  points. For a transitive TDS it is enough to show the existence of diam-mean equicontinuous points to say that the TDS is almost diam-mean equicontinuous. By Theorem \ref{Teo:Dicotomianormaldiammean} we have that  is not diam-mean sensitive. In general, there are TDSs with diam-mean equicontinuous points but still not almost diam-mean equicontinuous. The following CA illustrates this. \\

 Sensitivity on cellular automata has been studied in several papers (e.g. \cite{gilman1987classes,guckenheimer,garcia2016limit}).
We find it natural to ask if the previous dichotomy holds for the diam-mean notions on CA. We will show this has a false answer. 

Let $T:A^{\ZZ}\rightarrow A^{\ZZ}$ the CA from the previous subsection and $A_{3}=\{ a, b , c \}$. We define $T_{3}:A_{3}^{\ZZ}\rightarrow A_{3}^{\ZZ}$ as
\begin{displaymath}
\begin{array}{rcl}
T_{3}x_{i} & = & \left\lbrace \begin{array}{lcl}
a & if & x_{i}=a ; \\
b  & if & x_{i}=c ;\\
c & if & x_{i}=b.
\end{array}  
\right.  
\end{array}
\end{displaymath}

\noindent

We set $A_{S}:= A\times A_{3}.$ For every $x\in A_S$, we have that $x^1$ is the component on $A$ and $x^2$ is the component on $A_3$. 
Let $id:A_{3}^{\ZZ}\rightarrow A_{3}^{\ZZ}$ be the identity function and $T_{S}:A_{S}^\ZZ \rightarrow A_{S}^\ZZ$ a CA defined locally with

\begin{displaymath}
\begin{array}{rcl}
	
T_{S}x_{i} & = &\left\lbrace \begin{array}{lcl}
(Tx^{1}_{i},id(x^{2}_{i})) & \text{ if }  \gamma_{2}(x^{1}_{i})=\flechai
;\\	
	
(Tx^{1}_{i},T_{3}x^{2}_{i}) & \text{otherwise.}  %(\gamma_{2}(x^{1}_{i}),x^{2}_{i}) \in \{ ( \vacio , a ), ( \vacio , b ), ( \vacio , c ), ( \flechai , a ) \} ;\\

\end{array}

\right.  
\end{array}
\end{displaymath}
Thus, on $A$, $T_S$ behaves exactly as $T$, and on $A_3$, as $T_3$ except if there is an arrow on the first coordinate. When this happens, the periodicity on $b$ and $c$ changes phase. 
%\begin{lemma}\label{palabra sin fantasmas}
%Let $m>0$, $w\in A_{P}^{m}$, and $$x=\infinito(\vacio,\vacio).w(\cero, \vacio)(\vacio,\vacio)^{\infty}.$$ Then, there exists $N>0$ such that for all $n\geq N$ and all $0\leq i\leq |w|$,
%$$T_{P}^{n}x_{i}\in \{ (p,q): p\in \{ \vacio,\cero \}\wedge q\in A_{2} \} .$$

%\end{lemma}

%\begin{proof}
%This proof follows immediately from Lemma \ref{lemma3}.
%\end{proof}

%We want to show that $(A_{P}^\ZZ,T_{P})$ is not almost diam-mean equicontinuous. Using Proposition \ref{pro sigma}, we need to find a non-empty open set that does not contain any mean equicontinuity points.

In the next lemma, we will establish that the set of diam-mean equicontinuity points is not dense. 

\begin{lemma}\label{palabra S finito}
Let $m>0$ and $w\in A_{S}^{m}$ such that $w_{0}=(\vacio, b)$. Then, there exist $x,y\in A_{S}^{\ZZ}$ such that 
$$x_{[0,|w|-1]}=y_{[0,|w|-1]}=w$$ 
and the set  
$$\ZZ_{n\geq 0}\setminus \{n\in \ZZ_{n\geq 0} : T_{S}^{n}x_{0}\neq T_{S}^{n}y_{0} \}$$ 
is finite. 
\end{lemma}

\begin{proof}
Let $w\in A_{S}^{m}$ such that $w_{0}=(\vacio, b)$.  Let us define $$x=\infinito(\vacio, a ).w(\flechai , a )(\vacio, a)^{\infty}$$ and $$y= \infinito(\vacio,a).w(\vacio, a)^{\infty}.$$ Using Lemma \ref{palabrasinfantasmas} we can assume, without loss of generality (by waiting until the arrows are gone), that $$
w_{i}\in \{ (p,q) : p\in \{ \vacio, \cero , \uno , \dos\} \wedge q\in A_{3} \}
.$$ Now, there exists $N>0$ such that $T_{S}^{N}x_{0}=(\flechai, q)$, where $q\in \{ b, c \}$. We have two cases to prove.
\begin{itemize}
\item[ \ ] Case 1: $T_{S}^{N}x_{0}=(\flechai, b)$. 

This implies that $T_{S}^{N+1}x_{0}=(\vacio, b)$. Meanwhile,  $T_{S}^{N+1}y_{0}=(\vacio , c)$. Therefore, we can easily see that $T_{S}^{N+i}x_{0}\neq T_{S}^{N+i}y_{0}$, for all $i>0$.

\item[ \ ] Case 2: $T_{S}^{N}x_{0}=(\flechai,c)$.

 Again we have that $T_{S}^{N+1}x_{0}=(\vacio, c )$, so $T_{S}^{N+i}x_{0}\neq T_{S}^{N+i}y_{0}$ for all $i\geq 0$. 
\end{itemize}  
\end{proof}

%\begin{proof}
%This lemma follows immediately from Lemma \ref{palabra S finito}. 
%\end{proof}

Notice that for all $\e>0$, any $y\in B_{\e}(x)$, where $x_{0}=(\vacio,b)$, is not a diam-mean  equicontinuity point.

%\begin{theorem}
%	\label{teorema 2}
%$(A_{P}^{\ZZ}, T_{P})$ is neither mean sensitive and nor almost mean equicontinuous.
%\end{theorem}

%\begin{proof}
%Let us show that $(A_{P}^{\ZZ}, T_{P})$ is not mean sensitive, this is, for every $\e>0$ there exists a open set $U\subset A_{P}^{\ZZ}$ such that for every $x,y\in U$ we have that 
%$$\limsup_{n\rightarrow \infty} \frac{\sum_{i=0}^{n}d(T_{P}^{i}x,T_{P}^{i}y)}{n+1}< \e .$$
%From the Proposition \ref{punto-m-e-lemma}, we have that the element 
%$$x:=\cdots(\cero,\vacio) ( \vacio, \vacio)^{2^{1}} (\cero,\vacio) ( \vacio, \vacio)^{2^{0}}.(\cero,\vacio) ( \vacio, \vacio)^{2^{0}} (\cero,\vacio) ( \vacio, \vacio)^{2^{1}} \cdots$$ 
%is a mean equicontinuity point. From Proposition \ref{pro sigma}, for every $\e>0$, there exists $\d>0$ such that  for all $y,z \in B_{\d}(x)$ we have that
%$$\limsup_{n\rightarrow \infty} \frac{\sum_{i=0}^{n}d(T_{P}^{i}y,T_{P}^{i}z)}{n+1}< \e .$$
%Therefore, $(A_{P}^{\ZZ}, T_{P})$ is not mean sensitive.

%The fact that $(A_{P}^{\ZZ}, T_{P})$ is not almost mean equicontinuous  follows immediately from Lemma \ref{punto no m.e.}.  
%\end{proof}

\begin{theorem}
	\label{thm:2}
$(A_{S}^{\ZZ}, T_{S})$ is neither diam-mean sensitive and nor almost diam-mean equicontinuous.
\end{theorem}

\begin{proof}
From Lemma \ref{palabra S finito} we conclude that for every $x\in A_{S}^{\ZZ}$ such that $x_{0}=(\vacio,b)$, we have that $x$ is not a diam-mean equicontinuity point.  
	
From Proposition \ref{punto-d-m-e-propo}, we have that
$$x:= \infinito(\vacio,a) . (\vacio ,a) (\cero,a)  (\vacio,a) (\cero,a)^{2} (\vacio,a) (\cero,a)^{2^{2}} \cdots (\vacio,a)(\cero,a)^{2^{l}} \cdots$$
is a diam-mean equicotinuity point. Hence, for all $\e>0$ there exists $l\geq 0$ such that
$$\limsup_{n\to \infty}\frac{\sum_{j=0}^{n}\diam (T_{S}^{j}(B_{l+\sum_{i=0}^{l}2^{i}}(x))) }{n+1}<\e .$$
Therefore, $(A^{\ZZ}_{S},T_{S})$ is not diam-mean sensitive. 
\end{proof}

\bibliographystyle{plain}
\bibliography{ref}

\medskip

\begin{itemize}
	\item \emph{L. de los Santos Ba\~nos, Instituto de F\'isica, Universidad Aut\'onoma de San Luis Potos\'i, Mexico luguis.sb.25@gmail.com}
	\item \emph{F. Garc\'ia-Ramos, CONACyT \& Instituto de F\'isica, Universidad Aut\'onoma de San Luis Potos\'i, Mexico fgramos@conacyt.mx}
\end{itemize}
\end{document}